\documentclass[12pt,twoside]{article}

\usepackage{scrextend}
\usepackage[margin=1in]{geometry}
\usepackage{setspace}
\usepackage{mathtools}
\usepackage{amsfonts}
\usepackage{amssymb}
\usepackage{amsthm}
\usepackage{graphicx}
\usepackage{tikz}
\usepackage{hyperref}
\usepackage{caption}
\usepackage{subcaption}
\usepackage[title]{appendix}

\newcommand{\vep}{\varepsilon}
\newcommand{\vphi}{\varphi}
\newcommand{\R}{\mathbb{R}}

\newtheorem{theorem}{Theorem}[section]

\newtheorem{lemma}[theorem]{Lemma}

\newtheorem{claim}[theorem]{Claim}

\newtheorem{definition}[theorem]{Definition}

\title{Multi-Agent Recurrent Rendezvous Using Drive-Based Motivation}
\author{Craig Thompson and Paul Reverdy}

\begin{document}
	\maketitle
	
	\begin{abstract}
		Recent papers have introduced the Motivation Dynamics framework, which uses bifurcations to encode decision-making behavior in an autonomous mobile agent. In this paper, we consider the multi-agent extension of the Motivation Dynamics framework and show how the framework can be extended to encode persistent multi-agent rendezvous behaviors. We analytically characterize the bifurcation properties of the resulting system, and numerically show that it exhibits complex recurrent behavior suggestive of a strange attractor.
	\end{abstract}

	\section{Introduction}\label{sec:intro}
	
	Consider a group of people monitoring wildlife in a large area. Each person has a different location where they monitor the nest or den of some local animal, say a songbird. An individual will record the bird, make notes on its behavior, and otherwise gather information. Then, once sufficient information has been gathered, all members of the group rendezvous with each other to compare notes. Then, the group splits up again to gather more information. Abstractly, this is a problem of recurrent rendezvous alternating with designated tasks. In this paper we focus on developing a control scheme to automatically stabilize such a recurrent rendezvous behavior for a system of autonomous mobile robots.
	
	The multi-agent rendezvous problem sits at the intersection of robot motion planning and distributed control. In the problem, multiple autonomous robots must coordinate and arrive at a location in space which is not predefined. There may be constraints on what an agent knows about the space or the other agents. Additional challenges such as obstacles, adversaries, or noise may be present.
	
	The multi-agent rendezvous problem has seen significant attention in recent decades \cite{alpern1995rendezvous,francis2016flocking,lin2005multi,dimarogonas2007rendezvous}, with applications such as coverage and exploration \cite{li2016underwater,roy2001collaborative}, persistent recharging \cite{mathew2013graph,mathew2015multirobot}, and simultaneous location and mapping (SLAM) \cite{zhou2006multi}. Analyzing rendezvous is typically done using graph based approaches \cite{mathew2013graph,meghjani2012multi} (where one views the agents as nodes in a communication network, and analyzes various graph metrics) and dynamical systems tools \cite{friudenberg2018mobile}. 
	
	We distinguish two cases of the rendezvous problem. First is the case of one-time rendezvous, or non-recurrent rendezvous. Here the rendezvous problem is solved once the agents reach consensus. That is, if the problem begins at time $t_0$, the problem is solved if there is some time $t^*>t_0$ when the positions of the agents have reached (or are close enough to) a  point in the state space (possibly arbitrary). The second is the case of recurrent rendezvous. Here agents must always eventually rendezvous. That is, the recurrent problem is solved if for \emph{every time} $t>t_0$ there exists some later time $t^*>t$ when the agents reach a common point in the state space. There is some work on recurrent rendezvous, such as addressing the issue of autonomous vehicles needing to repetitively recharge \cite{mathew2013graph,mathew2015multirobot}, however most attention is given to the non-recurrent case. In this paper we focus on the recurrent case.
	
	When designing controllers to execute repetitive tasks, one approach is for an individual agent to use a decision making model to coordinate these tasks. Here we consider a task based approach to the problem of persistent multi-agent rendezvous. Task based autonomous behavior has been widely studied \cite{garrett2021integrated,ghallab2016automated,karpas2020automated}: complex behaviors are constructed from smaller, simpler tasks. Each individual task is encoded in a vector field (e.g., going to a target state is encoded as an attracting fixed point on the state space). The goal of the autonomous controller is to then coordinate these tasks to achieve said high-level behavior. Ultimately, this coordination can be framed as a form of decision making, which is a large and active area of research. 
	
	An individual in a decision making model has an action space $\mathcal{A} $, and state space $\mathcal{S}$. If $a_t\in\mathcal{A}$ and $s_t\in\mathcal{S}$ are the action and state at time $t$ we can define the history as the tuple $H_t = (s_0,a_0,s_1,\dots,a_{t-1},s_t)$. A policy $\rho$ is a map $\rho:H_t \to a_t$. In some cases we have a memory-less policy. That is, a policy that only depends on the current state $s_t$. While we have used notation that implies discrete time, this generalizes to the continuous time case in a natural way. For our applications it will be most reasonable to consider the continuous time case, and that will be the assumption from here on.
	
	In our task-based framework an action is choosing which of $n$ tasks to pursue. It is convenient for us to phrase our policy as an optimization problem. We consider an expanded state space $ s=(x,v) $ where $x\in\R^d$ is our physical state and $v\in\R_+^n$ is our value state. The value state encodes the relative values of pursuing each of the $n$ tasks. Our policy is then to pick the action which coincides with the task with maximal value. When framing the action/policy/state system as a control system, it is useful to introduce a motivation state $m\in\mathcal{M}$ which tracks the current task being pursued. Here $\mathcal{M}$ is the space of possible motivation states, and it may be discrete or continuous.
	
	Most decision making research reduces the problem to discrete decision states: Given $n$ tasks, pick exactly one to pursue. This can be formulated mathematically. Consider a motivation vector $ m \in \{0,1\}^n $ such that $ m_i = 1 $ iff $ \arg\max_j(v_j) = i $, and $ m_j = 0 $ otherwise. We call this the \emph{discrete motivation} case. A natural extension, and one that has been of recent interest, is the \emph{continuous motivation} case. The difference is that we consider the motivation vector $ m $ to live on the $ (n-1) $-simplex $\Delta^n$, where $\Delta^n = \{ x \in \R^n : x_i \geq 0, \sum_i x_i = 1 \}$. Then, following the work in \cite{DP-etal:13, PBR-DEK:18, PBR-VV-DEK:21}, we use a pitchfork bifurcation to encode a dynamical operation analogous to the $\arg \max$ operation from the discrete case. Following \cite{PBR-DEK:18}, we term the resulting dynamical system \emph{motivation dynamics}. Considering a continuous motivation state for our tasks produces interesting dynamics in the control system that are not otherwise seen in the discrete case. Moreover, it is possible to make entire system smooth, allowing for the use of smooth function analysis, as well as direct numerical integration.
	
	The work we present here is complementary to other recent work using bifurcation theory to construct decision-making mechanisms, e.g., \cite{RG-etal:18, AF-MG-NEL:19, PBR:20, AB-etal:21}. In contrast to these works, here we focus not on the decision-making mechanism itself, but on connecting the mechanism to physically-embedded tasks, i.e., the requirement for the various agents to move and perform rendezvous in response to their decisions.
	
	The main contribution of this paper is to address the persistent rendezvous problem using task-based navigation with a continuous motivation state. In the process, we develop a networked dynamical system that exhibits stable recurrent rendezvous encoded by an attractor. Concretely, our contributions are two-fold. First we develop the multi-agent extension of a control scheme originally developed in \cite{PBR-DEK:18} and expanded upon in \cite{thompson2019drive}. The scheme utilizes the so-called unfolding pitchfork bifurcation to coordinate agents with binary tasks. We then adapt this framework to the persistent rendezvous problem, and analyze the resulting dynamical system in a highly symmetric case. We derive analytical guarantees of the system behavior, and show that it achieves persistent rendezvous under appropriate conditions. In Theorem \ref{thm:deadlock-lim} we consider a limited form of the system and extract a closed-form critical parameter value for the breaking of system deadlock. Then, in Theorem \ref{thm:deadlock-lower-bound} we view the full system, and using results from the previous theorem, we prove the existence of such critical parameter values, along with a guaranteed bound for breaking deadlock. Next, by varying parameters $ \lambda $ and $ \sigma $ (defined below), we  characterize distinct regimes of behavior for the system (see Figure \ref{fig:behavioral-regimes}). Finally we discuss the results and provide a list of further research directions to be considered.
	
	
	
	
	
	

	

	\begin{center}
		\includegraphics[scale=1]{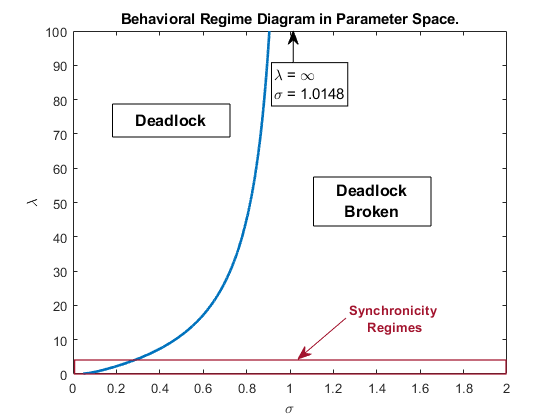}
		\captionof{figure}{Plot of behavioral regimes. Theorem \ref{thm:deadlock-lim} addresses the $ \lambda=\infty $ deadlock case, and Theorem \ref{thm:deadlock-lower-bound} addresses the deadlock bifurcation line for finite $ \lambda $.}\label{fig:behavioral-regimes}
	\end{center}

	\section{Preliminaries}\label{sec:prelims}
	We use the Motivation Dynamics framework \cite{PBR-DEK:18}. Here, we extend the Motivation Dynamics framework to the $N$-agent case. We assume that each agent $k$ has physical state $x_k \in \R^d$ and that its physical dynamics are fully actuated, i.e., that $\dot x_k = u$, with $x_k, u \in \R^d$. We assume that each agent has two tasks, one of which is to travel to a designated location, and the other is to rendezvous with the other agents. The values associated with these two tasks for agent $k$ are encoded in $v_k \in \R_+^2$ The two-task assumption allows us to reduce the motivation space, $\Delta^2$, to an interval on the real line. It is convenient to consider the motivation state $m_k$ of an agent to live on the interval $[-1,1]$, where each extreme of the interval uniquely corresponds with one of the two tasks. Let the state of agent $ k $ be denoted by
	$ \xi_k = (x_k^T,m_k,v_k^T)^T $, where $ x_k\in\R^{d} $ is the agent's physical location (or configuration), $ m_k \in [-1,1] $ is the agent's motivation state, and $ v_k\in\R_+^2 $ is the agent's value state. We assume that the agent's physical dynamics are fully-actuated, i.e.,  Denote $ \Xi $ as the matrix
	
	\begin{equation}\label{eq:matrix-form}
	\Xi = \left(\xi_1 \Big\vert \cdots \Big\vert \xi_N\right) = 
	\begin{pmatrix}
	X \\ 
	M \\ 
	V 
	\end{pmatrix},
	\end{equation}
	where $ X = \left(x_1 \vert \cdots \vert x_N\right) $, $ M = (m_1,\dots,m_k) $, and $ V = \left(v_1 \vert \cdots \vert v_N\right) $.
	
	For the time being, we will assume that the agents have two tasks, where their value dynamics are only influenced by their value state and the physical configuration of all agents. In this case agent $ k $ is controlled by the closed loop, fully continuous dynamical system
	
	\begin{subequations} \label{eq:expanded_dzk}
		\begin{align}
		\dot{x}_k &= f_{x_k}(X,M) = \frac{1+m_k}{2}F_{k,1}(X) + \frac{1-m_k}{2}F_{k,2}(X), \label{eq:dxk}\\
		\dot{m}_k &= f_{m_k}(M,V) = \sigma_k(1-m_k^2)(m_k+3\alpha_k), \label{eq:dmk}\\
		\dot{v}_k &= f_{v_k}(V,X) = \lambda_k(\vphi_k(X)-v_k), \label{eq:dvk}
		\end{align}
	\end{subequations}
	where $ \alpha_k = (v_{k,1}-v_{k,2})/(v_{k,1}+v_{k,2}) $. Compactly we have
	\begin{equation}\label{eq:dzk}
	\dot{\xi}_k = f_{\xi_k}(\Xi),
	\end{equation}
	and in matrix form
	\begin{equation}\label{eq:dXi}
	\dot{\Xi} = g(\Xi). 
	\end{equation}
	We will also denote the dynamics of the matrices $ X,M,V $ by $ \dot{X} = f_X(X,M) $, $ \dot{M} = f_M(M,V) $, $ \dot{V} = f_V(V,X) $.
	
	\subsection{Rendezvous with Designated Tasks}\label{sec:designated}
	By convention we will have task 1 be the designated task of a given agent, and task 2 will be rendezvous. We choose point visitation as the designated task. We have
	\[ \vphi_{k,1}(X) = \frac{1}{2}\left\| x_k-x^*_k \right\|^2, \]
	and $ F_{k,1}(X) = -\nabla \vphi_{k,1}(X) = -(x_k-x^*_k) $, where $x^*_k$ is the designated point for agent $k$ to visit.
	
	For the rendezvous task we choose navigation to the centroid. We set
	\[ \vphi_{k,2}(X) = \frac{N-1}{2N}\left\| x_k-\frac{1}{N-1}\sum_{\substack{j=1 \\ j\neq k}}^{N}x_j \right\|^2, \] 
	which gives
	\[ F_{k,2}(X) = -\nabla \vphi_{k,2}(X) = -\frac{N-1}{N}\left(x_k-\frac{1}{N-1}\sum_{j\neq k}x_j \right). \]
	Thus our first task looks like navigation to a designated point, and our second task looks like navigation to the centroid of the other agents.
	
	\section{The Symmetric Case}\label{sec:sym-case}
	Assuming $ N $ agents and $ x_k\in\R^2 $, denote $ R $ as the $ 2\times 2 $ rotation matrix given by
	\begin{equation}\label{eq:2d-rot-mat}
	R = \begin{pmatrix}
	\cos\left(\frac{2\pi}{N}\right) & -\sin\left(\frac{2\pi}{N}\right) \\
	\sin\left(\frac{2\pi}{N}\right) & \cos\left(\frac{2\pi}{N}\right)
	\end{pmatrix}. 
	\end{equation}
	We set the designated task locations as $ x_k^* = R^{k-1}u $, where $ u = (1,0)^T $ (i.e. $ N $ symmetrically  distributed points on the unit circle). At this point it is prudent to make the change of coordinates $ y_k = R^{-(k-1)}x_k $. Denoting the transformed $ X $ matrix by $ Y $, this results in
	\[ \vphi_{k,1}(Y) = \frac{1}{2}\left\| y_k-u \right\|^2, \]
	\[ \vphi_{k,2}(Y) = \frac{N-1}{2N}\left\| y_k-\frac{1}{N-1}\sum_{\substack{j=1 \\ j\neq k}}^{N}R^{j-k}y_j \right\|^2, \] 
	and 
	\[ \dot{y}_k = f_{y_k}(Y,M) = \frac{1+m_k}{2}(u-y_k) + \frac{1-m_k}{2}\left(\frac{1}{N}\sum_{j\neq k}R^{j-k}y_j - \frac{N-1}{N}y_k \right). \]
	
	Moreover, we will consider symmetric time scales for the motivation and value dynamics. That is, we set $ \sigma_k = \sigma $ and $ \lambda_k = \lambda $ for all $ k $. Let us set $ z_k = (y_k^T,m_k,v_k^T) $ and
	\[ Z = \left(z_1 \Big\vert \cdots \Big\vert z_N\right) = 
	\begin{pmatrix}
	Y \\ 
	M \\ 
	V 
	\end{pmatrix} \in \R^{5 \times N}. \]
	
	Then we may compactly represent our symmetric-case system by the dynamical equation
	\begin{equation}\label{eq:dZ}
	\dot{Z} = f(Z)
	\end{equation}
	
	The fact that this is the symmetric case allows us to utilize a group invariance of the system. First we note that following definition of equivariance for a map.
	
	\begin{definition}\label{def:equivariance}
		Consider a map $ h:\mathcal{A}\to \mathcal{A} $ and a group $ \Gamma $ with a well defined group action on elements $ x\in\mathcal{A} $. We say $ h $ is \emph{equivariant} under $ \Gamma $ if for all $ \gamma\in\Gamma $ and $ x\in\mathcal{A} $ we have that
		\begin{equation}\label{eq:equivariance}
		h(\gamma \cdot x) = \gamma\cdot h(x)
		\end{equation}
	\end{definition}
	
	With equivariance in mind, now we establish a symmetry group for the dynamics \eqref{eq:dZ}.
	
	\begin{lemma}\label{lemma:equivariant-group}
		Let $ C_n $ be the cyclical group generated by the permutation $ \gamma:(1,2,\dots,n) \mapsto (n,1,2,\dots,n-1) $ (i.e. the ``shift-by-one" permutation on $ n $ elements). Let $ \gamma $ act on a matrix $ M\in\R^{m\times n} $ by permuting the columns of $ M $ in the same way. Then we have that the mapping $ f $ given in \eqref{eq:dZ} is equivariant under $ C_N $ for $ N\geq3 $.
	\end{lemma}
	
	\begin{proof}
		Let $ \gamma $ represent the shift-by-one group action described above. Since $ \gamma $ generates $ C_N $ it suffices to prove that
		\[ \gamma\cdot f(Z) = f(\gamma\cdot Z). \]
		Let $ \gamma $ map the $ k $-th column of $ Z $ to $ \gamma(k) $. We have
		\[ \gamma\cdot f_{y_k}(Y,M) = f_{y_{\gamma(k)}}(Y,M) = \frac{1+m_{\gamma(k)}}{2}(u-y_{\gamma(k)})+ \frac{1-m_{\gamma(k)}}{2}\left(\frac{1}{N}\sum_{j\neq \gamma(k)}R^{j-\gamma(k)}y_j-\frac{N-1}{N}y_{\gamma(k)}\right) \]
		and
		\[ f_{y_k}(\gamma\cdot Y,\gamma\cdot M) = \frac{1+m_{\gamma(k)}}{2}(u-y_{\gamma(k)})+ \frac{1-m_{\gamma(k)}}{2}\left(\frac{1}{N}\sum_{j\neq k}R^{j-k}y_{\gamma(j)}-\frac{N-1}{N}y_{\gamma(k)}\right) \]
		Thus we must show that
		\[ \sum_{j\neq \gamma(k)}R^{j-\gamma(k)}y_j = \sum_{j\neq k}R^{j-k}y_{\gamma(j)}. \]
		One notes that $ (\gamma(j)-\gamma(k))\mod N = (j-k)\mod N $, and thus $ R^{\gamma(j)-\gamma(k)} = R^{j-k} $. Then by re-indexing the sum we may write
		\[ \sum_{j\neq \gamma(k)}R^{j-\gamma(k)}y_j = \sum_{j\neq \gamma(k)}R^{\gamma^{-1}(j)-k}y_j = \sum_{j\neq k}R^{j-k}y_{\gamma(j)}. \]
		This gives us $ \gamma\cdot f_{Y}(Y,M) = f_{Y}(\gamma\cdot Y,\gamma\cdot M) $. For $ f_V(V,Y) $ it suffices to show $ \vphi_{\gamma(k)}(Y) = \vphi_k(\gamma Y) $. This follows by applying the same argument as the one given above for $ f_Y $. Finally, one can see by inspection that $ \gamma\cdot f_{M}(M,V) = f_{M}(\gamma\cdot M,\gamma\cdot V) $. Thus we have that $ \forall \gamma'\in C_N $
		\[ \gamma'\cdot f(Z) = f(\gamma'\cdot Z). \]
	\end{proof}
	
	\begin{definition}\label{def:fixed-point-subspace}
		Consider a vector space $ \mathcal{V} $ with a well defined group action for elements of a group $ \Gamma $. We define the \emph{fixed point subspace} of a subgroup $ G\in\Gamma $ by
		\begin{equation}\label{eq:fixed-point-subspace}
		\text{\emph{Fix}}[G] = \left\{ x\in\mathcal{V}\vert~\gamma\cdot x = x,~ \forall\gamma\in G \right\}.
		\end{equation}
	\end{definition}
	
	We note that the fixed point subspace of $ C_N $ is the one generated by matrices $ Z $ with all columns the same (i.e. $ z_i = z_j~\forall i,j $). We will refer to this subspace as the \emph{symmetric subspace}. We will see later that said subspace possesses some very interesting properties.
	
	\subsection{Deadlock}\label{sec:sym-deadlock}
	The symmetric case displays many interesting behaviors for different $ (\sigma,\lambda) $ parameter regimes, and some of these boundaries between regimes admit direct analysis. One such boundary is that of the deadlock case.
	
	\begin{definition}\label{def:deadlock}
		We define deadlock as an equilibrium point $ Z_* $ of the system \eqref{eq:dZ} such that $ Z_* $ is asymptotically stable with $ |m_k| < 1 $ for all $ k $. (That is, all motivation variables are in a state of ``indecision".)
	\end{definition}
	
	Indeed, if we search for deadlock on the symmetric subspace, we find a candidate equilibrium.
	
	\begin{lemma} \label{lemma:deadlock}
		Let $ N\in\mathbb{N} $ and consider $ N\geq3 $. There is a potential deadlock equilibrium point on the symmetric subspace of \eqref{eq:dZ} which we will denote $ Z_* = (Y_*^T,M_*^T,V_*^T)^T $, given by
		\begin{equation}\label{eq:Z-equilibrium}
		y_k = (y^*,0)^T,\quad m_k = 2y^* - 1,\quad v_{k,1} = \frac{1}{2}(y^*-1)^2,\quad v_{k,2} = \frac{N-1}{2N}(y^*)^2, \quad \forall k=1,\dots,N 
		\end{equation}
		where $ \beta = y^* \in (0,1) $ solves
		\begin{equation}\label{eq:ystar-cubic}
		(2N-1)\beta^3 - (3N-1)\beta^2 - (N-1)\beta + N-1 = 0.
		\end{equation}
	\end{lemma}
	
	\begin{proof}
		We are considering the symmetry subspace, and so set $ z_1=z_2=\cdots z_N $. In order to be a deadlock equilibrium, $ |m_k|<1 $ by Definition \ref{def:deadlock}. Thus, we must have that $ m_k = -3(v_{k,1}-v_{k,2})/(v_{k,1}+v_{k,2}) $. Moreover, we have that $ v_{k,1} = \vphi_{k,1}(Y) $ and $ v_{k,2} = \vphi_{k,2}(Y) $. Substituting both into $ f_{y_k}(Y,M) = 0 $ will give the cubic in \eqref{eq:ystar-cubic}. Back substituting will give the equilibrium values for $ m_k $ and $ v_k $. 
		
		We must verify that there is exactly one root of \eqref{eq:ystar-cubic} in the interval $ (0,1) $. If we input $ \beta = -1,0,1,2 $ into the cubic we get $ -3N,N-1,-N,3N-3 $ respectively. We see that for $ N \geq 3 $ there are three consecutive changes in sign of the cubic. This implies that the three roots of the cubic are all real, and lie in the intervals $ (-1,0) $, $ (0,1) $, and $ (1,2) $. Roots which are greater than 1 and less than 0 will invalidate the $ |m|<1 $ assumption. Thus we must take the root in $ (0,1) $. \end{proof}
	
	A detailed numerical search suggests that the deadlock equilibrium described in Lemma \ref{lemma:deadlock} is unique. On the basis of this numerical evidence we make the following claim.
	
	\begin{claim}
		Let $N \geq 3$ and let $Z_*$ be the deadlock equilibrium defined in Lemma \ref{lemma:deadlock}. Then $Z_*$ is the unique equilibrium on the symmetry subspace.
	\end{claim}
	
	To study the stability of our system we must still consider a Jacobian matrix that is $ 5N\times 5N $. To get a better sense of the structure of the Jacobian, and how we might determine stability conditions, we will first consider the dynamics in the singularly perturbed limit of $ \lambda \to \infty $.
	
	\subsubsection{The $ \lambda \to \infty $ Limit}\label{sec:lam-limit}
	We split our system, denoting $ W = (Y^T,M^T)^T $ as our slow manifold variables, and the value state $V$ as our fast manifold variable. If we then take the $\lambda\to\infty$ limit we have that
	\[ v_{k,1} = \vphi_{k,1}(Y) \quad\text{and}\quad v_{k,2} = \vphi_{k,2}(Y). \]
	This results in the slow manifold system
	\begin{subequations} \label{eq:dzk-lim}
		\begin{align}
		\dot{y}_k &= f_{y_k}(Y,M) =  \frac{1+m_k}{2}(u-y_k) + \frac{1-m_k}{2}\left(\frac{1}{N}\sum_{j\neq k}R^{j-k}y_j - \frac{N-1}{N}y_k \right), \label{eq:dyk-lim}\\
		\dot{m}_k &= f_{m_k}(m_k,Y) = \sigma(1-m_k^2)\left(m_k+3\frac{\vphi_{k,1}(Y)-\vphi_{k,2}(Y)}{\vphi_{k,1}(Y)+\vphi_{k,2}(Y)}\right). \label{eq:dmk-lim}
		\end{align}
	\end{subequations}
	We will compactly represent the system by
	\begin{equation}\label{eq:dW}
	\dot{W} = \tilde{f}(W).
	\end{equation}
	and we denote the corresponding deadlock point for the reduced system by
	\begin{equation}\label{eq:W-equilibrium}
	W_* = (Y_*^T,M_*^T)^T,    
	\end{equation}
	where $ Y_* $ and $ M_* $ are still the same as in \eqref{eq:Z-equilibrium}. In Appendix \ref{apdx:limit-case-jacobian} we derive that the Jacobian $ D\tilde{f} $ can be represented compactly, and that  by \eqref{eq:det-Df-simple-limit} we may write the characteristic polynomial as
	\begin{equation}\label{eq:char-poly-lim}
	\det(D\tilde{f}-\mu I_{3N}) = \det\left(\tilde{M} + \sum_{k=1}^{N}\tilde{D}_k adj(\tilde{A}-\tilde{B}-\mu I_3)\tilde{C}_k\right)\det\left(\tilde{A}-\tilde{B}-\mu I_3\right)^{N-2} 
	\end{equation}
	This leads us to establish our first result
	\begin{theorem}\label{thm:deadlock-lim}
		There exists a critical value $ \sigma = \sigma^* $ for which the singularly-perturbed system \eqref{eq:dW} experiences a bifurcation which breaks the stability of the deadlock point \eqref{eq:W-equilibrium}:
		\begin{equation}\label{eq:deadlock-sigma-limit}
		\sigma^* \coloneqq \frac{1}{4y^*(1-y^*)}.
		\end{equation}
	\end{theorem}
	
	\begin{proof}
		From Lemma \ref{lemma:char-poly-limit}, we have a compact form for the characteristic polynomial given by
		\[ \det(D\tilde{f}-\mu I_{3N}) = (G_1(\mu;\sigma)+G_2(\mu;\sigma))^2(G_1(\mu;\sigma))^{N-2}, \]
		where $ G_1 $ and $ G_2 $ are given in the statement of Theorem \ref{lemma:char-poly-limit}. We note that both $ G_1(\mu;\sigma)+G_2(\mu;\sigma) $ and $ G_1(\mu;\sigma) $ are cubic in $ \mu $. Thus, one can write closed-form expressions for their roots. The expression for the roots of $ G_1(\mu;\sigma)+G_2(\mu;\sigma) $ is very complicated, and does not admit direct analysis of the critical $ \sigma $ value for bifurcation. However, $ G_1(\mu;\sigma) $ factors nicely across $ \mu $. Specifically, it may be written as
		\[ (\mu+1)\left[ (\mu+1)(4\sigma y^*(1-y^*)-\mu) - \frac{2\sigma (1-y^*)^2(1+y^*)}{\vphi_1^* + \vphi_2^*}\right], \]
		where $ \vphi_i^* = \vphi_{k,i}(Y_*) $ for $ i=1,2 $ are the task function values at the deadlock point. Direct computation of the roots gives
		\[ \begin{aligned}
		\mu_1 &= -1 \\
		\mu_{2,3} &= \frac{1}{2}\left( 4\sigma y^*(1-y^*) - 1 \pm \sqrt{(4\sigma y^*(1-y^*) - 1)^2 + 16\sigma y^*(1-y^*) - \frac{8\sigma (1-y^*)^2(1+y^*)}{\vphi_1^* + \vphi_2^*}} \right)
		\end{aligned} \]
		If we can show that 
		\[ 16\sigma y^*(1-y^*) - \frac{8\sigma (1-y^*)^2(1+y^*)}{\vphi_1^* + \vphi_2^*} < 0 \]
		then we can guarantee that $ \mu_{2,3} $ have negative real parts for $ \sigma < 1/(4y^*(1-y^*)) $. Noting that $8 \sigma y^* (1-y^*) > 0$ can be factored from the expression above, it suffices to show
		\[ 2y^* - \frac{1-(y^*)^2}{\vphi_1^* + \vphi_2^*} < 0, \]
		which is the case, as $ \vphi_i^* < 1/4 $ for $ i=1,2 $ and $ y^* < 1/2 $ (this can be checked by looking for sign switches in the left-hand-side of \eqref{eq:ystar-cubic} on the interval $ (0,1/2) $).
		
		Next we must show that the eigenvalues are complex, i.e.
		\[ (4\sigma y^*(1-y^*) - 1)^2 + 16\sigma y^*(1-y^*) - \frac{8\sigma (1-y^*)^2(1+y^*)}{\vphi_1^* + \vphi_2^*} < 0  \]
		for $ \sigma = 1/(4y^*(1-y^*)) $. It suffices to show
		\[ 2 - \frac{1-(y^*)^2}{y^*(\vphi_1^* + \vphi_2^*)} < 0. \]
		This once again holds due to the fact that $ \vphi_i^* < 1/4 $ for $ i=1,2 $ and $ y^* < 1/2 $.
		
		It remains to show the non-tangency condition of our roots, i.e., that the roots pass through the imaginary axis. The real part of the eigenvalues is given by $(4\sigma y^* (1- y^*) - 1)/2$, so the derivative with respect to $\sigma$ is $2 y^*(1-y^*)$. This is positive since $y^* \in (0, 1)$.

	\end{proof}

	\subsubsection{The case of finite $ \lambda $}\label{sec:lam-relax}
	It is possible to derive an explicit form for the characteristic polynomial of the relaxed system, but that polynomial would still be of degree 5, and in general would require numerical computations to find the roots. Indeed, given a fixed value of $\lambda$, we can numerically approximate the deadlock-breaking value of $ \sigma $, and vice versa. The results of such a computation are shown in Figure \ref{fig:crit-sigma-vs-lambda}. In place of deriving an exact form for the characteristic polynomial, we prove a lower bound on $ \sigma $ that guarantees deadlock-breaking.
	
	\begin{theorem}\label{thm:deadlock-lower-bound}
		There exists a finite lower bound $ \hat{\sigma} $ such that $ \sigma > \hat{\sigma} $ implies that the equilibrium point $ Z_* $ of the system \eqref{eq:dZ} is not stable. The bound is given by
		\begin{equation}\label{eq:deadlock-lower-bound}
		\hat{\sigma} = \frac{N\lambda + N + y^* - 1}{2Ny^*(1-y^*)}
		\end{equation}
		where $ y^* $ is the solution in the interval $ (0,1) $ to \eqref{eq:ystar-cubic}.
	\end{theorem}
	
	\begin{proof}
		By Lemma \ref{lemma:Df-trace} we have that
		\[ \text{tr}(Df) = 2\left( 1-y^* - N - N\lambda \right) + 4N\sigma y^*(1-y^*). \]
		The trace of a matrix is the sum of the eigenvalues, thus if the trace has positive real part there mus be at least one eigenvalue with positive real part. Setting $ \text{tr}(Df) \geq 0 $ and solving for $ \sigma $ gives the desired lower bound.
	\end{proof}
	
	\begin{figure}[h]
		\begin{center}
			\includegraphics[scale=0.8]{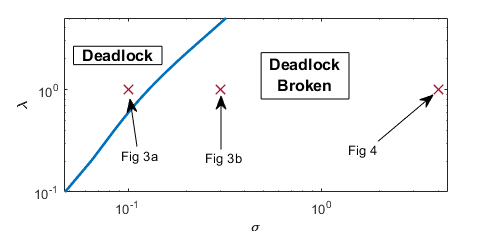}
		\end{center}
		\caption{The parameter space split into deadlock and non-deadlock regimes. Here we are considering $ N=3 $ agents.}\label{fig:crit-sigma-vs-lambda}
	\end{figure}
	
	As a demonstration we have included two examples of a subset of the state trajectories under deadlock and non-deadlock in Figure \ref{fig:deadlock-example}. We track the $ x_{k,1} $, $ m_k $ and $ \alpha_k $ states and we plot all three agents together at once. Here $ \alpha_k = (v_{k,1}-v_{k,2})/(v_{k,1}+v_{k,2}) $, which is the normalized difference between value states.
	
	\begin{figure}
		\centering
		\begin{subfigure}[b]{0.48\textwidth}
			\centering
			\includegraphics[width=\textwidth]{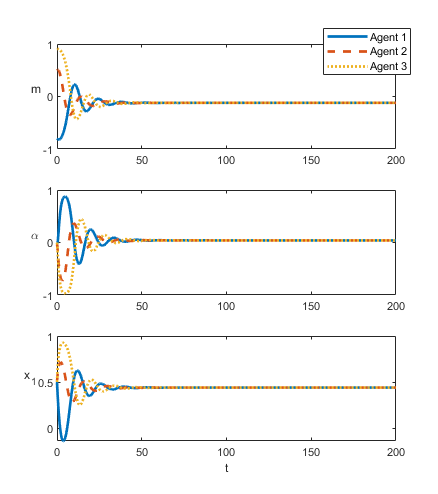}
			\caption{Deadlock ($\sigma = 0.1$)}
		\end{subfigure}
		\hfill
		\begin{subfigure}[b]{0.48\textwidth}
			\centering
			\includegraphics[width=\textwidth]{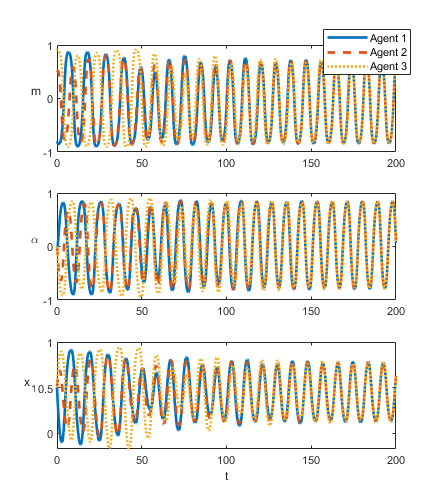}
			\caption{No Deadlock ($\sigma = 0.3$)}
		\end{subfigure}
		\caption{Comparison of deadlock and non-deadlock trajectories. In both cases $ \lambda = 1 $ and the initial conditions are the same.}\label{fig:deadlock-example}
	\end{figure}

	\subsection{Syncronicity and Asynchronicity}
	When we consider $ \sigma > \hat{\sigma} $, as given in \eqref{eq:deadlock-lower-bound}, and finite $ \lambda $ there are interesting dynamical regimes that emerge. Specifically, the system produces both synchronous and asynchronous rendezvous, dependent on initial conditions and parameter regimes. At this point, further direct analysis of the system becomes far less tractable, and yields very little information. Instead we will investigate these regimes numerically. For the remainder of this section we will consider the case $ N=3 $, as it is the simplest case that still experiences rich behavior.
	
	Once we are in the deadlock broken regime (See Figure \ref{fig:crit-sigma-vs-lambda}) the system can experience either synchronous or asynchronous rendezvous (See Figure \ref{fig:sync-vs-async}). Synchronicity for fixed $ \sigma $ and $ \lambda $ depends only on the initial conditions of the system. We should stress the system still experiences recurrent rendezvous, even in the asynchronous case. One way to see this is to consider the average distance to the centroid across all agents as it varies over time. We define the rendezvous metric by
	
	\begin{equation}\label{eq:rendezvous-metric}
	d(X(t)) \coloneqq \frac{1}{N}\sum_{k=1}^{N}\left\| x_k(t) - \frac{1}{N}\sum_{j=1}^N x_j(t) \right\|,
	\end{equation}
	
	or for the transformed variable $Y$ we have
	
	\begin{equation}\label{eq:rendezvous-metric-Y}
	\tilde{d}(Y(t)) \coloneqq \frac{1}{N}\sum_{k=1}^{N}\left\| y_k(t) - \frac{1}{N}\sum_{j=1}^N R^{j-k}y_j(t) \right\|.
	\end{equation}
	
	We plot two examples in Figure \ref{fig:sync-vs-async-metric}. 
	
	\begin{figure}
		\centering
		\begin{subfigure}[b]{0.48\textwidth}
			\centering
			\includegraphics[width=\textwidth]{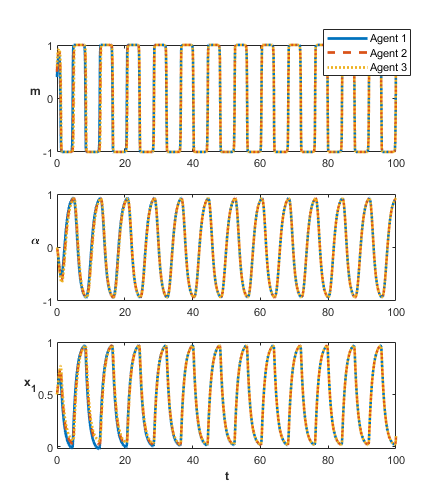}
			\caption{Synchronous Rendezvous}
		\end{subfigure}
		\hfill
		\begin{subfigure}[b]{0.48\textwidth}
			\centering
			\includegraphics[width=\textwidth]{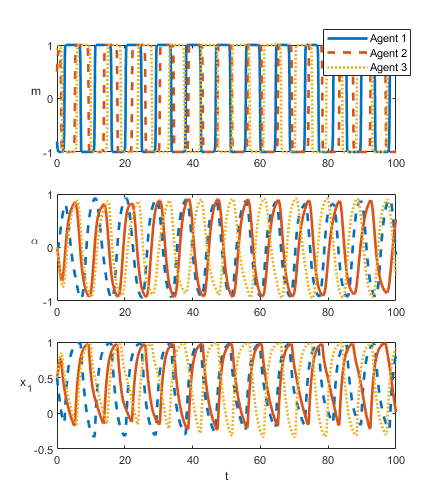}
			\caption{Asynchronous Rendezvous}
		\end{subfigure}
		\caption{Comparison of synchronous and asynchronous trajectories. In both cases $\sigma = 4$ and $ \lambda = 1 $. The only difference is in the initial conditions.}\label{fig:sync-vs-async}
	\end{figure}
	
	\begin{figure}
		\centering
		\begin{subfigure}[b]{0.48\textwidth}
			\centering
			\includegraphics[width=\textwidth]{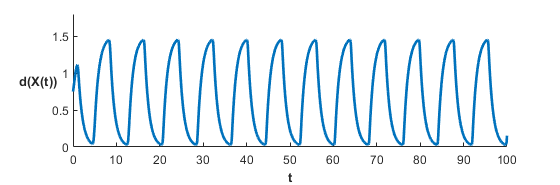}
			\caption{Synchronous Rendezvous}
		\end{subfigure}
		\hfill
		\begin{subfigure}[b]{0.48\textwidth}
			\centering
			\includegraphics[width=\textwidth]{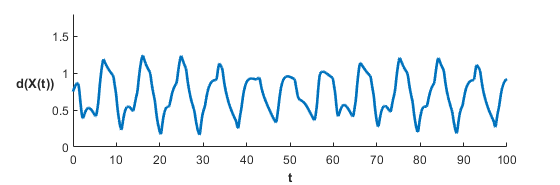}
			\caption{Asynchronous Rendezvous}
		\end{subfigure}
		\caption{Comparison of the rendezvous metric given in \eqref{eq:rendezvous-metric} for the synchronous and asynchronous trajectories from Figure \ref{fig:sync-vs-async}.}\label{fig:sync-vs-async-metric}
	\end{figure}
	
	Despite not performing as well as the synchronous case, the asynchronous case shows that the agents repetitively get close to one another. Under reasonably relaxed constraints this constitutes recurrent rendezvous. In fact, numerical evidence supports that there is an attractor of the system on which the asynchronous orbits live. The attractor itself is 15 dimensional, and so we can only plot slices of the state space. We consider the $ \alpha = (\alpha_1,\alpha_2,\alpha_3) $ space of the 3 agent system, recalling that $ \alpha_k = (v_{k,1}-v_{k,2})/(v_{k,1}+v_{k,2}) $, and the $ m = (m_1, m_2, m_3) $ space. Plotting trajectories in both spaces yields Figure \ref{fig:async-attractors}.
	
	\begin{figure}
		\begin{center}
			\includegraphics[width=0.48\linewidth]{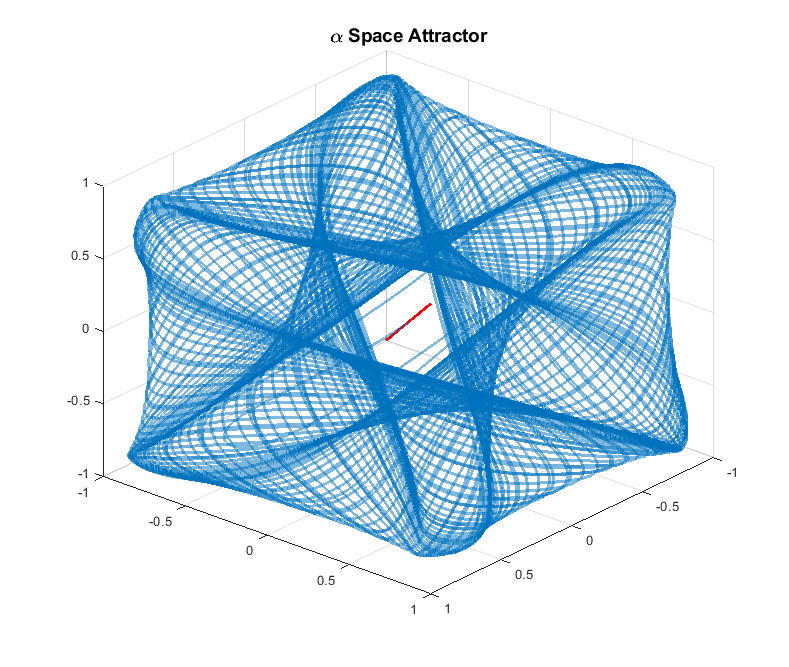} \hfill
			\includegraphics[width=0.48\linewidth]{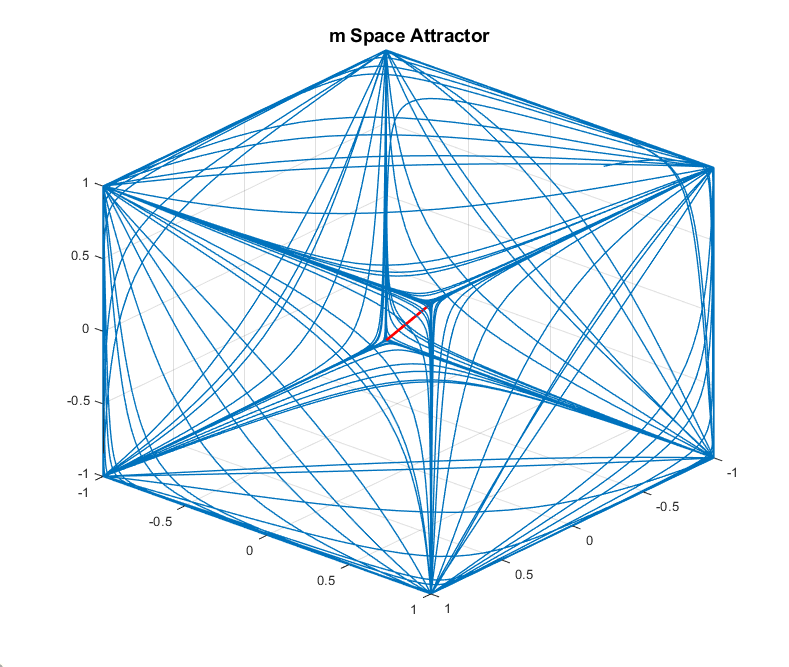}
		\end{center}
		\caption{Projections of the attractor that generates asynchronous trajectories. The left plot is a projection of the attractor into $ \alpha $-space and the right plot is a projection of the attractor into $ m $-space. For both plots the red line indicates the projection of the synchronous subspace (i.e. the $ z_1=z_2=z_3 $ subspace). Here $ \lambda = 1 $.}\label{fig:async-attractors}
	\end{figure}
	
	We can further characterize this attractor by taking a Poincar\'{e} section in the $ \alpha $ space. Consider the vector $ \alpha^* = (1,1,1) $. We construct our section by tracking when $ \alpha(t) \cdot \alpha^* = 0 $, i.e. when the trajectory punctures the plane defined by $ \alpha^* $. In addition, if we track the sign change we can see that two limit cycles emerge in the section. Let $ \tau_j \in [0,T] $ indicate the $ j $-th point in time such that $ \alpha(\tau_j) \cdot \alpha^* = 0 $. If we plot the even and odd indices separately, i.e. plot $ \alpha(\tau_{2\ell}) $ and $ \alpha(\tau_{2\ell-1}) $ for $ \ell = 1,2,\dots $, we see that two limit cycles emerge in the Poincar\'e section (See Figure \ref{fig:poincare}). These limit cycles further support the claim that the orbits are living on an attractor.
	
	\begin{figure}
		\begin{center}
			\includegraphics[width=0.6\linewidth]{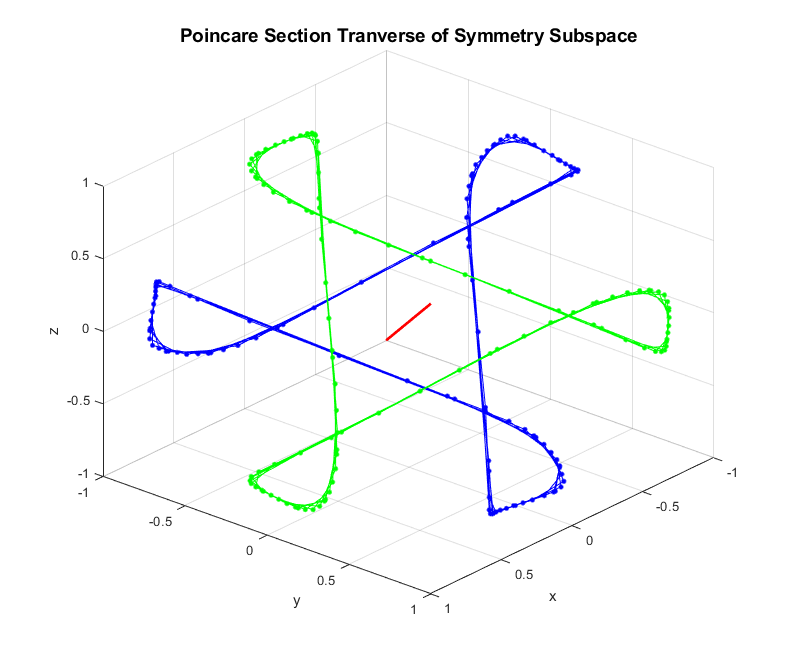} 
		\end{center}
		\caption{We plot the even and odd indices of $ \alpha(\tau_{j}) $. The points indicate the actual punctures of the Poincar\'e section, while the lines connecting them indicate adjacency in the ordering. Blue indicates the odd indices, and green indicates the even indices. The red line is the symmetry subspace $ \alpha_1 = \alpha_2 = \alpha_3 $. Here $ \lambda = 1 $.}\label{fig:poincare}
	\end{figure}
	
	This attractor leads us to make the following claim:
	
	\begin{claim}\label{claim:attractor}
		The system \eqref{eq:dZ} has an attractor $\mathcal{A}_{\sigma,\lambda} \subset \R^2\times[-1,1]^N\times\R_+^{2N}$ whose characteristics depend on the parameters $\sigma,\lambda$.
		
		Moreover, there exists a region of the parameter space such that $\mathcal{A}_{\sigma,\lambda}$ is not on the synchronous manifold, and is not a fixed point.
	\end{claim}
	
	This allows us to postulate a theorem.
	
	\begin{theorem}\label{thm:eventual_rendez}
		Assume that Claim \ref{claim:attractor} holds. Let $Z(t)$ be a solution to \eqref{eq:dZ}. Then fix $ \vep > 0 $, and consider the rendezvous metric from \eqref{eq:rendezvous-metric-Y}. If $Z(0) \in \mathcal{A}_{\sigma,\lambda}$ and there exists a time $ t_1 \geq 0 $ such that $ \tilde{d}(Y(t_1)) = \delta > 0 $, then there exists a time $ t_2 > t_1 $ such that $ \tilde{d}(Y(t_2)) < \delta + \vep $.
	\end{theorem}
	
	\begin{proof}
		An attractor $\mathcal{A}$ of a dynamical system has two necessary properties: (1) $\mathcal{A}$ is forward invariant, i.e. any trajectory starting in $\mathcal{A}$ stays in $\mathcal{A}$. (2) There is no proper subset of $\mathcal{A}$ which has this property. 
		
		Properties (1) and (2) imply given any $\vep>0$, any point $a\in\mathcal{A}$, and any trajectory $z$ such that $z(t_1)\in\mathcal{A}$, there is always a point in time $t_2>t_1$ such that $\|z(t_2)-a\| < \vep$. If no such point in time existed, then we could define the $\vep$-ball centered on $a$ by $B_\vep(a)$, and the set $\mathcal{A}\setminus B_\vep(a)$ would satisfy property (1), which would directly contradict property (2).
		
		Our claim in Theorem \ref{thm:eventual_rendez} then follows from the continuity of the rendezvous metric $\tilde{d}$: Given $\vep > 0$, there exists $\eta$ such that if $\| Z(t_2) - Z(t_1) \| < \eta$, continuity of $\tilde{d}$ implies that $|\tilde{d}(Y(t_2)) - \tilde{d}(Y(t_1))| < \vep$. By the preceding argument, such times $t_1,t_2$ must exist.
	\end{proof}
	
	The idea here is that if the orbit (which lives on this attractor) produces a small value $ \delta $ for the rendezvous metric, then there is a future time when the same orbit will pass close enough to the same point to produce a rendezvous metric that is sufficiently close to $ \delta $.
	
	\section{Discussion}
	In this paper, we have shown how the motivation dynamics architecture and associated results introduced in \cite{PBR-DEK:18} can be extended to the multi-agent case. In particular, we have demonstrated the feasibility of using the unfolding pitchfork decision-making mechanism to coordinate recurrent rendezvous tasks between autonomous agents.  Moreover, we have shown that the parameters of the multi-agent system \eqref{eq:dZ} can be chosen in such a way as to guarantee that the system does not reach a deadlock state, and that there will always eventually be a rendezvous event. In the language of formal methods \cite{HKG-ML-VR:18}, our system satisfies the following specification: always ((eventually rendezvous) and (eventually visit designated location)). The connection between the structure of the dynamical system and the discrete specification that it satisfies is beyond the scope of this paper, but is of great interest for future work.
	
	Several interesting research questions follow from the results of this paper. First, in the problem as it is presented in this paper, each agent has full knowledge of the others' positions. In other words, the communication graph is complete and communication is uncorrupted by noise or delays. The results presented in this paper could naturally be generalized to the case of an incomplete but connected communications graph, or the cases where communication among agents is corrupted by stochastic noise or communication delays. These generalizations should be feasible given standard tools from the distributed consensus literature.
	
	Second, the agents' tasks may be more subtle than the tasks considered here, where each agent has two tasks: exactly one designated point-attractor task and one rendezvous task where the agents must precisely rendezvous at a point. If approximate rendezvous, e.g., to a neighborhood, is acceptable, what further results might be possible? In the case that there are $M$ point-attractor tasks and $N \neq M$ agents, how can the system automatically distribute these tasks among the agents?
	
	Third, there is a fundamental mathematical question of characterizing the attractor whose existence we postulate in Claim \ref{claim:attractor}. It is likely that symmetric bifurcation theory \cite{MG-IS-DGS:88} will provide a set of tools to enable this analysis. 
	Answering these questions will push forward knowledge of how the unfolding pitchfork can be used to coordinate high-level autonomous behaviors.

	\bibliographystyle{abbrv}
	\bibliography{task-coord,rendezvous}
	
	\newpage
	
	\begin{appendices}

		\section{Jacobian in $ \lambda\to\infty $ Limit}\label{apdx:limit-case-jacobian}
		In this appendix, we compute the Jacobian for our system \eqref{eq:dZ} in the singular limit $\lambda \to \infty$. We also derive an expression for its characteristic polynomial.
		
		\begin{lemma}
			Consider the system \eqref{eq:dZ} in the limit $\lambda \to \infty$ given by \eqref{eq:dW}. The Jacobian of \eqref{eq:dW} at the equilibrium point \eqref{eq:W-equilibrium} has the form
			\begin{equation}\label{eq:Df-tilde}
			D\tilde{f} = \begin{pmatrix}
			\tilde{A} & \tilde{B}_{12} & \tilde{B}_{13} & \cdots \\
			\tilde{B}_{21} & \tilde{A} & \tilde{B}_{23} & \cdots \\
			\tilde{B}_{31} & \tilde{B}_{32} & \tilde{A} & \cdots \\
			\vdots & \vdots & \vdots & \ddots
			\end{pmatrix} 
			\end{equation}
			where
			\begin{equation}\label{eq:A-tilde}
			\tilde{A} = \begin{pmatrix}
			\frac{1-y^*-N}{N} & 0 & \frac{1}{2} \\
			0 & \frac{1-y^*-N}{N} & 0 \\
			\frac{8\sigma y^*(1-y^*)(2(y^*)^2-2y^*-1)}{\vphi_1^*+\vphi_2^*} & 0 & 4\sigma y^*(1-y^*) \\
			\end{pmatrix},
			\end{equation}with $ \vphi_1^* = \frac{1}{2}(1-y^*)^2 $, $ \vphi_2^* = \frac{N-1}{2N}(y^*)^2 $, and
			\begin{equation}\label{eq:Bij-tilde}
			\tilde{B}_{ij} = \begin{pmatrix}
			\frac{1-y^*}{N}R^{j-i}  & 0_{2\times1} \\
			-\frac{8\sigma y^*(1-y^*)(y^*-2)}{(N-1)(\vphi_1^*+\vphi_2^*)}(y^*,0) R^{j-i}
			& 0
			\end{pmatrix}, \quad\text{with}\quad R^k = \begin{pmatrix}
			\cos\left(\frac{2\pi k}{N}\right) & -\sin\left(\frac{2\pi k}{N}\right) \\
			\sin\left(\frac{2\pi k}{N}\right) & \cos\left(\frac{2\pi k}{N}\right)
			\end{pmatrix}.
			\end{equation}  
		\end{lemma}
		
		\begin{proof}
			Direct computation.
		\end{proof}
		
		\subsection{Computation of Characteristic Polynomial of $ D\tilde{f} $}\label{sec:char-poly-limit}
		
		First we state a lemma which allows us to compactly represent the Jacobian of \eqref{eq:dW} at the equilibrium point.
		
		\begin{lemma}\label{lemma:compact-jacobian-limit}
			The Jacobian \eqref{eq:Df-tilde} of \eqref{eq:dW} at the equilibrium given by \eqref{eq:W-equilibrium} can be compactly represented by
			\begin{equation}\label{eq:compact-jacobian-limit}
			D\tilde{f} = I_N\otimes(\tilde{A}-\tilde{B}) + \tilde{V}^T\tilde{U},
			\end{equation}
			where $ \tilde{A} $ is given in \eqref{eq:A-tilde}, $~\otimes$ indicates the Kronecker product, and  $ \tilde{B} \coloneqq \tilde{B}_{ii} $, where $ \tilde{B}_{ij} $ is given in \eqref{eq:Bij-tilde} (note that $ \tilde{B}_{ii} $ has no dependence on $ i $). The block matrices $\tilde{V}$ and $\tilde{U}$ are  given by $ \tilde{V} \coloneqq (\tilde{C}_1^T~\tilde{C}_2^T~\cdots~\tilde{C}_N^T) $ and $ \tilde{U} \coloneqq (\tilde{D}_1~\tilde{D}_2~\cdots~\tilde{D}_N) $, where the components $\tilde{C_i}$ and $\tilde{D_i}$ are given by
			\[ \tilde{C}_{i} \coloneqq \begin{pmatrix}
			\frac{1-y^*}{N}R^{-i} & 0_{2\times1} \\
			-\frac{8\sigma y^*(1-y^*)(y^*-2)}{(N-1)(\vphi_1^*+\vphi_2^*)}(y^*,0) R^{-i} & 0
			\end{pmatrix}, \qquad 
			\tilde{D}_j \coloneqq \begin{pmatrix}
			R^{j} & 0_{2\times1}\\
			0_{1\times2} & 0
			\end{pmatrix}. \]
		\end{lemma}
		
		\begin{proof}
			Observe that $ \tilde{C}_i\tilde{D}_j = \tilde{B}_{ij} $. Then have that
			\[ \begin{pmatrix}
			\tilde{B}_{11} & \tilde{B}_{12} & \tilde{B}_{13} & \cdots \\
			\tilde{B}_{21} & \tilde{B}_{22} & \tilde{B}_{23} & \cdots \\
			\tilde{B}_{31} & \tilde{B}_{32} & \tilde{B}_{33} & \cdots \\
			\vdots & \vdots & \vdots & \ddots
			\end{pmatrix} = \tilde{V}^T\tilde{U}. \]
			Thus it follows that
			\[ \tilde{D}\tilde{f} = \begin{pmatrix}
			\tilde{A} & \tilde{B}_{12} & \tilde{B}_{13} & \cdots \\
			\tilde{B}_{21} & \tilde{A} & \tilde{B}_{23} & \cdots \\
			\tilde{B}_{31} & \tilde{B}_{32} & \tilde{A} & \cdots \\
			\vdots & \vdots & \vdots & \ddots
			\end{pmatrix} = I_N\otimes(\tilde{A}-\tilde{B}) + \tilde{V}^T\tilde{U}. \]
		\end{proof}
		
		We can compute the characteristic polynomial by using the representation given in \eqref{eq:compact-jacobian-limit} and the Matrix Determinant Lemma, which we will state here without proof.
		
		\begin{lemma}\label{lemma:mat-det-lemma}
			Let $ A $ be an $ n\times n $ invertible matrix, and let $ U $ and $ V $ be $ m\times n $ matrices. Then
			\begin{equation}\label{eq:mat-det-lemma}
			\det(A + V^TU) = \det(I_m + UA^{-1}V^T)\det(A).
			\end{equation}
		\end{lemma}
		
		With this we may reduce the determinant of $ D\tilde{f}-\mu I_{3N} $ to the product of determinants of $ 3\times3 $ matrices.
		
		\begin{lemma}\label{lemma:det-reduction-limit}
			Let $ N\in\mathbb{N} $ and consider $ N \geq 3 $. Let $ \tilde{A},\tilde{B},\tilde{C}_k,\tilde{D}_k,\tilde{V},\tilde{U} $ be defined the same as in Lemma \ref{lemma:compact-jacobian-limit}. Then we may write the characteristic polynomial of $ D\tilde{f} $ as
			\begin{equation}\label{eq:det-Df-simple-limit}
			\det(D\tilde{f}-\mu I_{3N}) = \det\left(\tilde{M} + \sum_{k=1}^{N}\tilde{D}_k adj(\tilde{A}-\tilde{B}-\mu I_3)\tilde{C}_k\right)\det\left(\tilde{A}-\tilde{B}-\mu I_3\right)^{N-2},
			\end{equation}
			where $ \tilde{M} $ is given by
			\[ \tilde{M} \coloneqq \begin{pmatrix}
			\det(\tilde{A}-\tilde{B}-\mu I_3) & 0 & 0 \\
			0 & \det(\tilde{A}-\tilde{B}-\mu I_3) & 0 \\
			0 & 0 & 1 & \\
			\end{pmatrix}. \]
		\end{lemma}
		
		\begin{proof}
			We have that
			\begin{equation*}
			\begin{aligned}
			\det(D\tilde{f}-\mu I_{3N}) &= \det\left(I_N\otimes(\tilde{A}-\tilde{B}-\mu I_3) + \tilde{V}^T\tilde{U}\right) \\
			&= \det\left(I_3 + \tilde{U}(I_N\otimes(\tilde{A}-\tilde{B}-\mu I_3))^{-1}\tilde{V}\right)\det\left(I_N\otimes(\tilde{A}-\tilde{B}-\mu I_3)\right) \\
			&= \det\left(I_3 + \tilde{U}(I_N\otimes(\tilde{A}-\tilde{B}-\mu I_3)^{-1})\tilde{V}\right)\det\left(\tilde{A}-\tilde{B}-\mu I_3\right)^N \\
			&= \det\left(I_3 + \sum_{k=1}^{N}\tilde{D}_k(\tilde{A}-\tilde{B}-\mu I_3)^{-1}\tilde{C}_k\right)\det\left(\tilde{A}-\tilde{B}-\mu I_3\right)^N \\
			&= \det\left(I_3 + \tfrac{1}{\det(\tilde{A}-\tilde{B}-\mu I_3)}\sum_{k=1}^{N}\tilde{D}_k adj(\tilde{A}-\tilde{B}-\mu I_3)\tilde{C}_k\right)\det\left(\tilde{A}-\tilde{B}-\mu I_3\right)^N \\
			&= \det\left(\tilde{M} + \sum_{k=1}^{N}\tilde{D}_k adj(\tilde{A}-\tilde{B}-\mu I_3)\tilde{C}_k\right)\det\left(\tilde{A}-\tilde{B}-\mu I_3\right)^{N-2}.
			\end{aligned}
			\end{equation*}
			The second equality in the derivation comes from the Matrix Determinant Lemma, and the final equality comes from the fact that the determinant of an $ n\times n $ matrix is $ n $-linear in the rows of the matrix and that $ \tilde{D}_kadj(\tilde{A}-\tilde{B}-\mu I_3)\tilde{C}_k $ is all 0 below the second row.
		\end{proof}
		
		Before computing the characteristic polynomial we need to evaluate the expression $ \sum_{k=1}^{N}\tilde{D}_k adj(\tilde{A}-\tilde{B}-\mu I_3)\tilde{C}_k $. To do this we need the following lemma.
		
		\begin{lemma}\label{lemma:rot-mat-automorphism-lemma}
			Let $ N\in\mathbb{N} $ and consider $ N \geq 3 $. Define $ R $ to be the same as in \eqref{eq:2d-rot-mat}. Then we have
			\begin{equation}\label{eq:rot-mat-automorphism-lemma}
			\sum_{k=1}^{N}R^k\begin{pmatrix}a & 0 \\ 0 & b\end{pmatrix}R^{-k} = \frac{N}{2}\begin{pmatrix}a+b & 0 \\ 0 & a+b\end{pmatrix}.
			\end{equation}
		\end{lemma}
		\begin{proof}
			Denote $ \theta_k = \frac{2\pi k}{N} $. We have that
			\[ \begin{aligned}
			R^k\begin{pmatrix}1 & 0 \\ 0 & 0\end{pmatrix}R^{-k} = \begin{pmatrix} \cos^2\theta_k & \sin\theta_k\cos\theta_k \\ \sin\theta_k\cos\theta_k & \sin^2\theta_k\end{pmatrix} = \frac{1}{2}\begin{pmatrix} 1+\cos(2\theta_k) & \sin(2\theta_k) \\ \sin(2\theta_k) & 1-\cos(2\theta_k)\end{pmatrix}.
			\end{aligned} \]
			We also have that
			\[ \sum_{k=1}^{k=N}\cos(2\theta_k) = \sum_{k=1}^{k=N}\sin(2\theta_k) = 0. \]
			This comes from the fact that the sum of roots of unity is 0. That is
			\[ \sum_{k=1}^{k=N}\exp\left(\frac{i2\pi mk}{N}\right) = 0 \qquad \text{for } m = 1,\dots,N-1. \]
			Thus 
			\[ \sum_{k=1}^{N}R^k\begin{pmatrix}1 & 0 \\ 0 & 0\end{pmatrix}R^{-k} = \frac{N}{2}\begin{pmatrix}1 & 0 \\ 0 & 1\end{pmatrix}. \]
			Similarly one can show that
			\[ \sum_{k=1}^{N}R^k\begin{pmatrix}0 & 0 \\ 0 & 1\end{pmatrix}R^{-k} = \frac{N}{2}\begin{pmatrix}1 & 0 \\ 0 & 1\end{pmatrix}. \]
			Thus we have that
			\[ \sum_{k=1}^{N}R^k\begin{pmatrix}a & 0 \\ 0 & b\end{pmatrix}R^{-k} = a\sum_{k=1}^{N}R^k\begin{pmatrix}1 & 0 \\ 0 & 0\end{pmatrix}R^{-k} + b\sum_{k=1}^{N}R^k\begin{pmatrix}0 & 0 \\ 0 & 1\end{pmatrix}R^{-k} = \frac{N}{2}\begin{pmatrix}a+b & 0 \\ 0 & a+b\end{pmatrix}. \]
		\end{proof}
		
		Now we are in a place to compute the characteristic polynomial.
		
		\begin{lemma}\label{lemma:char-poly-limit}
			The characteristic polynomial of $ D\tilde{f} $ is given by
			\begin{equation}\label{eq:char-poly-limit}
			(G_1(\mu;\sigma)+G_2(\mu;\sigma))^2(G_1(\mu;\sigma))^{N-2},
			\end{equation}
			where 
			\[ G_1(\mu;\sigma) = (\mu+1)\left[ (\mu+1)(4\sigma y^*(1-y^*)-\mu) - \frac{2\sigma (1-y^*)^2(1+y^*)}{\vphi_1^* + \vphi_2^*}\right] \]
			and 
			\[ G_2(\mu;\sigma) = \frac{1-y^*}{2}(g_1(\mu;\sigma)+g_3(\mu;\sigma)) - \frac{4N\sigma (1-y^*)(y^*-2)(y^*)^2}{(N-1)(\vphi_1^* + \vphi_2^*)}g_2(\mu;\sigma) \]
			with
			\[ \begin{aligned}
			g_1(\mu;\sigma) &= -(1+\mu)\left[ 4\sigma y^*(1-y^*) - \mu \right] \\
			g_2(\mu;\sigma) &= \frac{1+\mu}{2} \\
			g_3(\mu;\sigma) &= -(1+\mu)\left[ 4\sigma y^*(1-y^*) - \mu \right] + \frac{4\sigma (1-y^*)^2(1+y^*)}{\vphi_1^* + \vphi_2^*} 
			\end{aligned}  \]
		\end{lemma}
		
		\begin{proof}
			The adjugate of $ \tilde{A}-\tilde{B}-\mu I_3 $ has the form
			\[ adj(\tilde{A}-\tilde{B}-\mu I_3) = \begin{pmatrix}
			g_1(\mu;\sigma) & 0 & g_2(\mu;\sigma) \\
			0 & g_3(\mu;\sigma) & 0 \\
			g_4(\mu;\sigma) & 0 & g_5(\mu;\sigma) 
			\end{pmatrix} \]
			where $ g_1 $, $ g_2 $, and $ g_3 $ are given as in the theorem above. The expressions for $ g_4 $ and $ g_5 $ are not necessary, as one can compute
			\[ \tilde{D}_k adj(\tilde{A}-\tilde{B}-\mu I_3)\tilde{C}_k = \begin{pmatrix}
			\tilde{Q}_k & \begin{array}{c} 0 \\ 0 \end{array} \\
			\begin{array}{cc} 0 & 0 \end{array} & 0
			\end{pmatrix} \]
			where
			\[ \tilde{Q}_k = R^k\left[ \frac{1-y^*}{N} \begin{pmatrix} g_1(\mu;\sigma) & 0 \\ 0 & g_3(\mu;\sigma) \end{pmatrix} - \frac{4N\sigma (1-y^*)(y^*-2)y^*}{(N-1)(\vphi_1^* + \vphi_2^*)} \begin{pmatrix} y^*g_2(\mu;\sigma) & 0 \\ 0 & 0 \end{pmatrix} \right]R^{-k} \]
			Applying Lemma \ref{lemma:rot-mat-automorphism-lemma} gives us $ \sum_{k=1}^N \tilde{Q}_k = I_2 G_2(\mu;\sigma) $.
			
			Now note that $ G_1(\mu;\sigma) $ is precisely $ \det(\tilde{A}-\tilde{B}-\mu I_3) $. If we combine these derivations with the form of $ \det(D\tilde{f}-\mu I_{3N}) $ given in \eqref{eq:det-Df-simple-limit} we arrive at the desired expression:
			\[ \det(D\tilde{f}-\mu I_{3N}) = (G_1(\mu;\sigma)+G_2(\mu;\sigma))^2(G_1(\mu;\sigma))^{N-2}. \]
		\end{proof}

		\section{Jacobian for finite $ \lambda $}\label{apdx:general-case-jacobian}
		In this appendix we consider the full system \eqref{eq:dZ}. Using similar procedure as in Appendix \ref{apdx:limit-case-jacobian} we a simplified expression for the characteristic polynomial and trace of the Jacobian of \eqref{eq:dZ} at the equilibriup point \eqref{eq:Z-equilibrium}.
		
		\begin{lemma}
			Consider the system equilibrium given by \eqref{eq:Z-equilibrium}. The Jacobian of \eqref{eq:dZ} at that equilibrium has the form
			\begin{equation}\label{eq:Df}
			Df = \begin{pmatrix}
			A & B_{12} & B_{13} & \cdots \\
			B_{21} & A & B_{23} & \cdots \\
			B_{31} & B_{32} & A & \cdots \\
			\vdots & \vdots & \vdots & \ddots
			\end{pmatrix} 
			\end{equation}
			where there are $ 5N $ rows and columns, and $ A $ and $ B_{ij} $ are $ 5\times5 $ matrices. Specifically we have
			\begin{equation}\label{eq:A}
			A = \begin{pmatrix}
			\frac{1-y^*-N}{N} & 0 & \frac{1}{2} & 0 & 0 \\
			0 & \frac{1-y^*-N}{N} & 0 & 0 & 0 \\
			0 & 0 & 4\sigma y^*(1-y^*) & \frac{24\sigma y^*(1-y^*)\vphi_2^*}{(\vphi_1^*+\vphi_2^*)^2} & -\frac{24\sigma y^*(1-y^*)\vphi_1^*}{(\vphi_1^*+\vphi_2^*)^2} \\
			\lambda(y^*-1) & 0 & 0 & -\lambda & 0 \\
			\lambda y^* & 0 & 0 & 0 & -\lambda
			\end{pmatrix},
			\end{equation}
			with $ \vphi_1^* = \frac{1}{2}(1-y^*)^2 $, $ \vphi_2^* = \frac{N-1}{2N}(y^*)^2 $, and
			\begin{equation}\label{eq:Bij}
			B_{ij} = \begin{pmatrix}
			\frac{1-y^*}{N}R^{j-i} & 0_{2\times3} \\
			0_{2\times2} & 0_{2\times3} \\
			\frac{-\lambda}{N-1}(y^*,0) R^{j-i}
			& 0_{1\times3}
			\end{pmatrix}, \quad\text{with}\quad R^k = \begin{pmatrix}
			\cos\left(\frac{2\pi k}{N}\right) & -\sin\left(\frac{2\pi k}{N}\right) \\
			\sin\left(\frac{2\pi k}{N}\right) & \cos\left(\frac{2\pi k}{N}\right)
			\end{pmatrix}.
			\end{equation}
		\end{lemma}
		
		\begin{proof}
			Direct computation.
		\end{proof}

		\subsection{Computation of Characteristic Polynomial of $ Df $}\label{sec:char-poly}
		
		\begin{lemma}\label{lemma:compact-jacobian}
			The Jacobian of \eqref{eq:dZ} at that equilibrium given by \eqref{eq:Z-equilibrium} can be compactly represented by
			\begin{equation}\label{eq:compact-jacobian}
			Df = I_N\otimes(A-B) + V^TU.
			\end{equation}
			where ``$ ~\otimes $" indicates the Kronecker product. Here $ B = B_{ii} $, where $ B_{ij} $ is given by \eqref{eq:Bij} (note that $ B_{ii} $ has no dependence on $ i $). Additionally we have 
			\[ C_{i} \coloneqq \begin{pmatrix}
			\frac{1-y^*}{N}R^{-i} & 0_{2\times3} \\
			0_{2\times2} & 0_{2\times3} \\
			\frac{-\lambda}{N-1}(y^*,0) R^{-i}
			& 0_{1\times3}
			\end{pmatrix}, \qquad 
			D_j \coloneqq \begin{pmatrix}
			R^{j} & 0_{2\times3} \\
			0_{3\times2} & 0_{3\times3}
			\end{pmatrix}, \]
			and block matrices $ V \coloneqq (C_1^T~C_2^T~\cdots~C_N^T) $ and $ U \coloneqq (D_1~D_2~\cdots~D_N) $ 
		\end{lemma}
		
		\begin{proof}
			Observe that $ C_iD_j = B_{ij} $. Then have that
			\[ \begin{pmatrix}
			B_{11} & B_{12} & B_{13} & \cdots \\
			B_{21} & B_{22} & B_{23} & \cdots \\
			B_{31} & B_{32} & B_{33} & \cdots \\
			\vdots & \vdots & \vdots & \ddots
			\end{pmatrix} = V^TU. \]
			Thus it follows that
			\[ Df = \begin{pmatrix}
			A & B_{12} & B_{13} & \cdots \\
			B_{21} & A & B_{23} & \cdots \\
			B_{31} & B_{32} & A & \cdots \\
			\vdots & \vdots & \vdots & \ddots
			\end{pmatrix} = I_N\otimes(A-B) + V^TU. \]
		\end{proof}
		
		\begin{lemma}\label{lemma:det-reduction}
			Let $ N\in\mathbb{N} $ and consider $ N \geq 3 $. Let $ A,B,C_k,D_k,V,U $ be defined as in Lemma \ref{lemma:compact-jacobian}. Then we may write the characteristic polynomial of $ Df $ as
			\begin{equation}\label{eq:det-Df-simple}
			\det(Df-\mu I_{3N}) = \det\left(M + \sum_{k=1}^{N}D_k adj(A-B-\mu I_3)C_k\right)\det\left(A-B-\mu I_3\right)^{N-2},
			\end{equation}
			where $ M $ is given by
			\[ M = \begin{pmatrix}
			\det(A-B-\mu I_3) & 0 & 0 & 0 & 0 \\
			0 & \det(A-B-\mu I_3) & 0 & 0 & 0 \\
			0 & 0 & 1 & 0 & 0 \\
			0 & 0 & 0 & 1 & 0 \\
			0 & 0 & 0 & 0 & 1 \\
			\end{pmatrix}. \]
		\end{lemma}
		
		\subsection{Trace of $ Df $}\label{sec:Df-trace}
		\begin{lemma}\label{lemma:Df-trace}
			The trace of $ Df $ is given by
			\begin{equation}\label{eq:trace-Df}
			\text{Tr}(Df) = 2\left( 1-y^* - N - N\lambda \right) + 4N\sigma y^*(1-y^*).
			\end{equation}
		\end{lemma}
		
		\begin{proof}
			This follows from direct computation using \eqref{eq:Df} and $ \eqref{eq:A} $.
		\end{proof}
		
	\end{appendices}
\end{document}